\documentclass[12pt]{amsart}
\usepackage[colorlinks,linkcolor=blue,citecolor=blue,urlcolor=blue, pdfcenterwindow, pdfstartview={XYZ null null 1.2}, pdffitwindow, pdfdisplaydoctitle=true]{hyperref}
\usepackage[english]{babel}
\usepackage{array}
\usepackage{threeparttable}
\usepackage{rotating}
\usepackage[section]{placeins}
\usepackage{amssymb} 
\usepackage{mathrsfs}
\usepackage{enumerate}
\usepackage{color}
\usepackage{amsrefs}        
\usepackage[all]{xy} 

\input xy
\xyoption{all}
\newtheorem{lemma}{Lemma}[section]
\newtheorem{lemm}[lemma]{Lemma}

\newtheorem{coro}[lemma]{Corollary}
\newtheorem{theo}[lemma]{Theorem}

\begin{document} 
\title{Flat bundles with complex analytic holonomy}
\author{I. Chatterji}
\author{G. Mislin}
\author{Ch. Pittet}
\address{MAPMO Universit\'e d'Orl\'eans, Orl\'eans}    
\email{Indira.Chatterji@univ-orleans.fr}
\address{Department of Mathematics ETHZ, Z\"urich}    
\curraddr{Department of Mathematics, Ohio State University}
\email{mislin@math.ethz.ch}
\address{LATP UMR 7353 CNRS Aix-Marseille Universit\'e}
\email{pittet@cmi.univ-mrs.fr}
\date{August 6, 2013}
\begin{abstract}
Let $G$ be a connected complex Lie group. We show that any flat principal
$G$-bundle over any finite $CW$-complex pulls back to
a trivial bundle over some finite covering space of
the base space if and only if each  real characteristic class of positive degree of $G$ vanishes. A third equivalent condition is that the derived group of the radical
of $G$ is simply connected. As a corollary, the same conditions are equivalent if $G$ is a connected amenable  Lie group. In particular, if $G$ is a  connected compact  Lie group then any flat principal
$G$-bundle over any finite $CW$-complex pulls back to
a trivial bundle over some finite covering space of
the base space. 
\end{abstract}
\maketitle
\section{Introduction}
Let $G$ be a  Lie group. A principal $G$-bundle
over a connected $CW$-complex $X$ is called {\sl flat}, if there is a 
homomorphism $$\rho: \pi_1(X)\to G,$$ the holonomy of the flat bundle, such
that the given bundle is equivalent to the $G$-bundle $\tilde{X}\times_{\rho} G\to X$
canonically associated with the universal cover $\tilde X$ of $X$; the
notation $\tilde{X}\times_\rho G$ refers to the orbit space of $\tilde{X}\times G$
under the $\pi_1(X)$-action given by $$\gamma(x,g)=(\gamma x, \rho(\gamma)g).$$
Flat $G$-bundles are
characterized by the fact that the
classifying map $\theta: X\to BG$ factors as $$X\to B\pi_1(X)\to BG,$$ where
the first arrow classifies the universal cover of $X$ and the second one
is $B\rho$. Equivalently, if $G^{\delta}$ denotes the group $G$ with the discrete topology and $\iota:G^{\delta}\to G$ denotes the identity map, a principal $G$-bundle over $X$ is flat, if and only if it is  classified by a map $\theta: X\to BG$ which  factors through $$B\iota:BG^{\delta}\to BG.$$  We refer the reader to \cite{Kamber-Tondeur-LNM} for more details on the above facts. 

A principal $G$-bundle over $X$ is called {\sl virtually trivial}
if its pull-back to some finite covering space of $X$ is trivial. 

\emph{Under which conditions on a connected Lie group $G$ is any flat principal $G$-bundle, over any finite $CW$-complex, virtually trivial?}

A necessary condition is that each real characteristic class of  $G$ in $H^*(BG^\delta,\mathbb R)$, in the sense of \cite[p. 23]{Gromov}, of  positive degree vanishes; that is the map 
\[
	H^*(BG,\mathbb R)\to H^*(BG^{\delta},\mathbb R)
\]
induced by $B\iota$ is zero if $*>0$. This necessary condition is fulfilled if $G$ is a complex reductive
group; this follows from a result of Kamber and Tondeur \cite[Theorem 3.5]{Kamber-Tondeur}. A well-known result 
of Deligne and Sullivan states that any flat principal $GL(n,\mathbb C)$-bundle over any finite $CW$-complex is virtually trivial \cite{DS}.

Before we state our main result, we recall that the radical $R$ of a connected Lie group $G$ is its maximal connected normal
solvable subgroup. It is always a closed subgroup of $G$ but its commutator subgroup
$[R,R]$ is in general not closed in $G$.

\begin{theo}\label{thm: main} Let $G$ be a connected complex Lie group. 
	The following conditions are equivalent.
	\begin{enumerate}
		\item Any flat principal $G$-bundle over any finite $CW$-complex is virtually trivial.
		\item The map $H^*(BG,\mathbb R)\to H^*(BG^{\delta},\mathbb R)$ is zero in  positive degree.
		\item The map $H^2(BG,\mathbb R)\to H^2(BG^{\delta},\mathbb R)$ is zero.
		\item The derived subgroup $[R,R]$ of the radical $R$ of $G$ is simply connected.
	\end{enumerate}
\end{theo}
According to Got\^o \cite[Theorem 25]{Got}, a connected solvable Lie group $R$ is linear if and only if
the \emph{closure} of its derived subgroup is simply connected.
As the map between fundamental groups   
$$\pi_1\left([R,R]\right)\to\pi_1\left(\overline{[R,R]}\right),$$
induced by the inclusion, is one-to-one, it follows
that a connected complex Lie group $G$ whose radical is linear satisfies the equivalent
conditions of the theorem.
The chain of implications $(1)\Rightarrow(2)\Rightarrow(3)\Rightarrow(4)$ holds for 
\emph{any} connected Lie group. The proof of $(3)\Rightarrow(4)$ is given
in \cite[Proof of Theorem 2.2]{CCMP} and is based on a construction of Goldman \cite{Gol}. Hence, in order to prove the theorem, it is enough to
show that if $G$ is a connected complex Lie group and if the derived subgroup of the radical
of $G$ is simply connected then any flat principal $G$-bundle over any finite $CW$-complex is virtually trivial. 

The main steps in the proof are the following.
All real characteristic classes of a connected Lie group $G$ are bounded,
when viewed as classes in $H^*(BG^\delta,\mathbb R)$, if and only if
the derived subgroup of the radical of $G$ is simply connected \cite{CCMP}. Combining this fact
with Gromov's {\sl{Mapping Theorem}} \cite[Section 3.1]{Gromov}, we reduce the problem to
the case of semisimple groups. A connected complex semisimple Lie group $C$ has a unique complex
algebraic structure and there exists a Chevalley
integral group scheme $G_\mathbb{Z}$, whose set of $\mathbb{C}$-points
$G_\mathbb{Z}(\mathbb{C})$ in its Lie group topology, $G_\mathbb{Z}(\mathbb{C})_{\text{Lie}}$,
is as a complex Lie group isomorphic to $C$ (for the existence of $G_\mathbb{Z}$ see \cite{SGA};
see also \cite{FP}). As explained in \cite{FM},  this opens the way to the application
of Sullivan's completion techniques, in a similar way as in \cite{DS}: the Hasse principle
applies and solves the problem.

In Lemma \ref{amenablecase} below, we show that if $G$ is a connected amenable Lie group, then its  universal complexification $$\gamma_G: G\to G^+,$$  (see Section \ref{sec: proof of the corollary} below) is one-to-one and is a homotopy equivalence (these two properties on the universal complexification characterize
connected amenable Lie groups among connected Lie groups, but we won't need this fact).
As a consequence, Theorem \ref{thm: main} has the following corollary.
(Notice that none of the conditions in Theorem \ref{thm: main} refers to a complex structure on the Lie group.)

\begin{coro}\label{maincoro}
For a connected amenable Lie group,
the four conditions in Theorem  \ref{thm: main} are equivalent.
\end{coro}

The proof of Corollary \ref{maincoro} is explained in Section \ref{sec: proof of the corollary} below. In \cite{Gol}, Goldman proved $(1)\Leftrightarrow(4)$ for connected solvable Lie groups.
The finiteness
assumption on the $CW$-complex which is the base of the bundle is not stated explicitly in \cite{Gol} but is necessary,
as the following example shows: the flat principal $S^1$-bundle over $K(\mathbb{Q}/\mathbb{Z},1)$
with classifying map induced by the inclusion $\mathbb{Q}/\mathbb{Z}\subset S^1$
is not virtually trivial because
its classifying map $$K(\mathbb{Q/Z},1)\to BS^1=K(\mathbb{Z},2)$$ corresponds to an element
of infinite order in $$H^2(K(\mathbb{Q/Z},1),\mathbb{Z})\cong \varprojlim{\mathbb{Z}/n\mathbb{Z}}.$$
 That $(2)\Leftrightarrow(4)$ for connected solvable Lie groups follows from \cite[Theorem 1.1]{CCMP} because a solvable group with the discrete topology is amenable and the  real cohomology of amenable groups vanishes \cite{John}.

A compact Lie group is amenable as a Lie group, and  the radical 
of a compact Lie group is abelian, hence the above corollary implies the following.

\begin{coro} Let $G$ be a connected compact Lie group. Then 
any flat principal $G$-bundle over any finite $CW$-complex is virtually trivial.      
\end{coro}

A flat bundle whose holonomy is non-amenable and has no complex structure may fail
to be virtually trivial even if all its real characteristic classes vanish in positive degree.  
The cohomology ring $$H^*(BSO(2n+1),\mathbb R)$$ is generated by Pontrjagin classes
\cite[Theorem 15.9]{MilSta} and Pontrjagin classes of flat $GL(n,R)$-bundles vanish
\cite[Appendix C, Corollary 2]{MilSta} hence 
\[H^*(BSL(2n+1),\mathbb R)\to H^*(BSL^{\delta}(2n+1),\mathbb R)\]
vanishes in positive degree.
But  Millson \cite{Millson} and Deligne \cite{Deligne} have constructed, for each $n\geq 3$, flat principal $SL(n,\mathbb R)$-bundles
over finite $CW$-complexes which are not virtually trivial.

The paper is organized as follows.
In Section 2 we prove Theorem \ref{thm: main} for the case of a complex semisimple Lie group.
In Section 3
we prove a lemma closely related to Goldman's result \cite{Gol}, the main difference being
that it also applies to
bundles which are not necessarily flat.
In Section 4 we explain how the results of  the
previous sections imply Theorem \ref{thm: main} in full generality. In
Section 5 we prove Corollary \ref{maincoro}. 

\section{The complex semisimple case}
First, we fix some notation and recall some facts concerning Sullivan's
completion functor \cite{Sullivan}. Let $p$ be a prime. We can think of Sullivan's
$p$-adic completion as a functor $X\mapsto \hat{X}_p$ on the homotopy category of
connected $CW$-complexes, together with a natural transformation $X\to \hat{X}_p$
which for $X$ a simply connected $CW$-complex of finite type induces 
isomorphisms 
\[
 \pi_i(X)\otimes \hat{\mathbb{Z}}_p\to \pi_i(\hat{X}_p),\,  i\ge 2,
\] 
with
$\hat{\mathbb{Z}}_p$ denoting the ring of $p$-adic integers.
We will need the following basic fact.

\begin{lemm}\cite[Thm. 3.2]{Sullivan}.\label{Sullivan}
Let $X$ be a finite $CW$-complex and $Y$ a simply connected $CW$-complex of
finite type. A map 
$$f:X\to Y$$
 is homotopic to a constant map if and only if
for every prime $p$ the map 
$$\hat{f}_p: X\to Y\to\hat{Y}_p$$
is homotopic to a
constant map.
\end{lemm}
The point here is that the space $X$ in the lemma does not need to be simply connected
(or nilpotent).


\begin{lemm}\label{semisimple case}
Let $C$ be a connected complex semisimple Lie group and $X$ a connected finite
$CW$-complex. Let $P: E\to X$ be a flat principal $C$-bundle. Then $P$ is virtually
trivial.
\end{lemm}

\begin{proof}
We can assume that $C$ is isomorphic to $G_\mathbb{Z}(\mathbb{C})_\text{Lie}$
for some Chevalley group scheme $G_\mathbb{Z}$. 
Let $\psi: \overline{X}\to X\to BC$ be the classifying map for $P$ pulled back to
a finite covering space $\overline{X}$ of $X$. Because $BC$ is simply connected and
of finite type and $\overline X$ is a finite complex, we can prove that $\psi$ is homotopic to
a constant map for a particular $\overline{X}$ by showing that for every prime $p$, the map 
$\hat{\psi}_p: \overline
{X}\to \widehat{BC}_p$ into the $p$-adic completion of $BC$ is
homotopic to a constant map (Lemma \ref{Sullivan}). 
Let $\pi$ be the fundamental
group of $X$ and $\rho: \pi\to C= G_\mathbb{Z}(\mathbb{C})_{\text{Lie}}$
the holonomy of the bundle $P$. Since $\pi$
is finitely generated, there exist a subring $\Lambda\subset\mathbb{C}$ of
finite type over $\mathbb{Z}$ such that the image of $\rho$ is contained
in $G_\mathbb{Z}(\Lambda)$. Choose a maximal ideal $\mathfrak{m}\subset\Lambda$ such
that the finite field $\mathbb{F}=\Lambda/\mathfrak{m}$ has characteristic $q$ different
from any torsion prime occurring in the finite torsion subgroup of $\bigoplus_i H^i(X, \pi_{i-1}(C))$. Let $\overline{\mathbb{F}}$
be an algebraic closure of $\mathbb{F}$ and $H\subset\mathbb{C}$ a strict Henselization of $\Lambda$ in 
$\mathbb{C}$, with residue field ${\overline{\mathbb{F}}}$.
We then obtain a diagram of group homomorphisms
$$\xymatrix{
\pi\ar[r]\ar@/_/[drr]_\phi&G_\mathbb{Z}(\Lambda)\ar[r]&G_\mathbb{Z}(H)\ar[d]\ar[r]&G_\mathbb{Z}(\mathbb{C})\ar[r]&C\\
&&G_\mathbb{Z}(\overline{\mathbb{F}})&&
}$$
such that the image of the composite map $\phi:\pi\to G_\mathbb{Z}({\overline{\mathbb{F}}})$ is finite,
because $\pi$ is finitely generated and $G_\mathbb{Z}({\overline{\mathbb{F}}})$ is a locally finite group. Let $\overline{X}$
be the finite covering space of $X$ corresponding to the kernel of $\phi$. We will show that the bundle P
pulled back to $\overline{X}$ is trivial.  Let $\psi: \overline{X}\to BC$ be the
classifying map for that bundle. For every prime $\ell$ different from
the characteristic
$q$ of $\mathbb{F}$ the map $\hat{\psi}_\ell: \overline{X}\to \widehat{BC}_\ell$ is homotopically trivial, because up to
homotopy it can be factored through the homotopically trivial map $\overline{X}\to BG_\mathbb{Z}(H)\to
BG_\mathbb{Z}(\overline{\mathbb{F}})$, using natural maps
$$\xymatrix{\overline{X}\ar[r]&
BG_\mathbb{Z}(\Lambda)\ar[r]&BG_\mathbb{Z}(H)\ar[d]\ar[r]&BG_\mathbb{Z}(\mathbb{C})
\ar[r]&{\widehat{BC}}_\ell\ar[d]^\simeq\\
&&BG_\mathbb{Z}(\overline{\mathbb{F}})\ar[rr]&&((BG_{\overline{\mathbb{F}}})_{\text{et}})^{\hat{}}_\ell
}$$
%
For the maps and notation see page 432 of \cite{FM}. It remains to deal with the prime $\ell=q$: we
need to show that $\hat{\psi}_q:\overline{X}\to \widehat{BC}_q$ is homotopically trivial too.
Because the bundle $P$ is flat with connected complex reductive structure
group, the rational characteristic classes of $P$ are all zero (the map $H^*(BC,\mathbb{Q})\to
H^*(X,\mathbb{Q})$ is zero in positive degrees cf.~Theorem 3.5 of
Kamber-Tondeur \cite{Kamber-Tondeur}). As a result, the obstructions to trivializing
the $\hat{C}_q$-fibration classified by $X\to \widehat{BC}_q$, are all $q$-torsion.
These obstructions lie in the group $\bigoplus_i H^i({X},\pi_{i-1}(C)\otimes\hat{\mathbb{Z}}_q)$, but this
group is torsion-free by the choice of $q$. We conclude that the map $\hat{\psi}_q: \overline{X}\to
\widehat{BC}_q$ is homotopically trivial and the Hasse principle (Lemma \ref{Sullivan})
applied to the map $\psi: \overline{X}\to BC$
implies therefore that $\psi$ must be homotopically trivial too.

\end{proof}

\section{Bundles with solvable holonomy}
The following is a characterization of virtually trivial principal bundles
over finite connected $CW$-complexes, in case the structural group is a connected solvable
Lie group. It can be viewed as a variation of a theorem due to Goldman \cite{Gol}, but
without assuming that the bundle in question is flat.

\begin{lemm}{\label{R-case}}
Let $R$ be a solvable connected Lie group and $P: E\to X$ a principal
$R$-bundle over the connected $CW$-complex $X$, with $\psi: X\to BR$ the
classifying map. Assume that $H_1(X,\mathbb{Z})$ is finitely generated.
Then the bundle $P$ is virtually trivial
if and only if $\psi^*: H^2(BR,\mathbb{R})\to H^2(X,\mathbb{R})$
is the 0-map.
\end{lemm}

\begin{proof} 
If $\psi^*\neq 0$ then for any finite covering space $\pi:\overline{X}\to X$ the composition
$$H^2(BR,\mathbb{R})\stackrel{\psi^*}{\rightarrow} H^2(X,\mathbb{R})\stackrel{\pi^*}{\rightarrow} H^2(\overline{X},\mathbb{R})$$
is non-zero too, because $\pi^*: H^*(X,\mathbb{R})\to H^*(\overline{X},\mathbb{R})$ is injective.
It follows that $P$ cannot pull back to a trivial bundle on some finite covering
space of $X$. Conversely, assume that $\psi^*=0$. Because
$R$ is homotopy equivalent 
to a maximal compact subgroup $T\subset R$, $T$ a torus, 
$BR$ is homotopy equivalent to $K(\mathbb{Z}^n,2)$
where $\pi_1(T)\cong \mathbb{Z}^n$. It follows that there is a single obstruction
$\omega\in H^2(X,\pi_1(R))$ to the existence of a section for $P$. Because
of our assumption on $X$, the kernel of the natural map $H^2(X,\mathbb{Z})\to H^2(X,\mathbb{R})$ is 
finite (isomorphic to the torsion subgroup of $H_1(X,\mathbb{Z})$) and thus the
hypothesis that $\psi^*=0$
implies that $\omega$ must be a torsion class. From the
universal coefficient theorem we see that therefore
$$\omega\in {\text{Ext}}(H_1(X,\mathbb{Z}),\pi_1(R))\rightarrowtail H^2(X,\pi_1(R))\twoheadrightarrow
\text{Hom}(H_2(X,\mathbb{Z}),\pi_1(R)).$$
Let $\text{Tor}\subset H_1(X,\mathbb{Z})$ be the finite torsion subgroup and choose a surjection
$\theta: \pi_1(X)\to\text{Tor}$. Let $f: \overline{X}\to X$ denote the covering space corresponding
to the kernel of $\theta$. It follows that 
$$f^*(\omega)=0\in \text{Ext}(H_1(\overline{X},\mathbb{Z}),\pi_1(R))\,.$$ 
But $f^*(\omega)$ is
the only obstruction to the existence of a section for the principal $R$-bundle
$f^*P: f^*E\to \overline{X}$, showing that $f^*P$ is trivial and
thus completing the proof of the lemma.
\end{proof}


\section{Proof of Theorem \ref{thm: main}}

We will need the following two auxiliary results.

\begin{lemm}{\label{R}} Let $R$ be a solvable connected Lie group and
let $P: E\to Z$ be a principal $R$-bundle over the finite connected complex $Z$,
classified by $\kappa: Z\to BR$. Let $G$ be a connected Lie group containing
$R$ as a normal, closed subgroup and denote by $\iota:R\to G$ the
inclusion. Assume that the
principal $G$-bundle over $Z$ classified by $(B\iota)\circ \kappa :Z\to BG$
satisfies $\kappa^*\circ (B\iota)^*=0: H^2(BG,\mathbb{R})\to H^2(Z,\mathbb{R})$.
Then the principal $R$-bundle $P$ is virtually trivial.
\end{lemm}

\begin{proof}
Let $Q=G/R$. Since for any connected Lie group the
second homotopy group vanishes and the fundamental group is abelian,
we have a short exact sequence of abelian groups
$$0\to \pi_1(R)\to \pi_1(G)\to \pi_1(Q)\to 0,$$ inducing a split short exact sequence
of $\mathbb{R}$-vector spaces
$$0\to {\text{Hom}}(\pi_1(Q),\mathbb{R})\to {\text{Hom}}(\pi_1(G),\mathbb{R})\stackrel{\Phi}\to
{\text
{Hom}}(\pi_1(R),\mathbb{R})\to 0\,.$$
For any connected Lie group $L$, the group $H_2(BL,\mathbb{R})$ is naturally
isomorphic to $H_1(L,\mathbb{R})\cong \pi_1(L)\otimes\mathbb{R}$. It follows
that the natural map $(B\iota)^*:H^2(BG,\mathbb{R})\to H^2(BR,\mathbb{R})$ corresponds
to the surjective map $\Phi$. Therefore, the vanishing of

$$\kappa^*\circ (B\iota)^*: H^2(BG,\mathbb{R})\to H^2(Z,\mathbb{R})$$
implies the vanishing of 
$$\kappa^*: H^2(BR,\mathbb{R})\to H^2(Z,\mathbb{R})\,.$$
Using Lemma $\ref{R-case}$ we conclude that the principal
$R$-bundle
$P$ is virtually trivial.
\end{proof}

\begin{lemm}\label{G}
Let $G$ be a connected Lie group and let $R$ be its radical.
Suppose that its derived group $[R,R]$ is simply connected in its Lie
group topology and that $G/R$ has a finite fundamental group.
Let $G^\delta$ denote the group $G$ with the discrete
topology. Then the identity map on the underlying sets $ \iota_G:G^\delta\to G$
induces the zero map $\iota_G^*: H^2(BG,\mathbb{R})\to H^2(BG^\delta,\mathbb{R})$.
\end{lemm}

\begin{proof}
There is a short exact sequence of Lie groups
\begin{equation}
R\stackrel{\iota}{\rightarrow} G\stackrel{\pi}{\rightarrow} Q
\end{equation}
with $R$ the radical of $G$ and $Q$ semisimple.

{\sl Split case}. We first assume that the short exact sequence $(1)$
is split, with $\sigma:Q\to G$ a splitting. For a discrete group $D$
we write $H^*_b(D,\mathbb{R})$ for its bounded real cohomology and we 
denote by
$$\theta_D: H^*_b(D,\mathbb{R})\to H^*(D,\mathbb{R})$$
the forgetful map.
Because $R^\delta$ is an amenable discrete group, the inflation map
$$\pi_b^*: H^*_b(Q^\delta,\mathbb{R})\to H^*_b(G^\delta,\mathbb{R})$$ 
is an isomorphism (cf.~Ivanov \cite{Ivanov}*{Theorem 3.8.4}, see also 
Gromov's {\sl{Mapping Theorem}} \cite{Gromov}*{Section 3.1}).

Therefore, the induced maps
$$\pi^*_b: H^*_b(Q^\delta,\mathbb{R})\to H^*_b(G^\delta,\mathbb{R}),\quad  \sigma^*_b:
H^*_b(G^\delta,\mathbb{R})\to H^*_b(Q^\delta,\mathbb{R})$$
are inverse isomorphisms. We write
$$
\pi^*_\delta: H^*(Q^\delta,\mathbb{R})\to H^*(G^\delta,\mathbb{R}),\quad  \sigma_\delta^*:
H^*(G^\delta,\mathbb{R})\to H^*(Q^\delta,\mathbb{R})$$
and
$$\pi^*_{top}: H^*(BQ,\mathbb{R})\to H^*(BG,\mathbb{R}),\quad  \sigma^*_{top}:
H^*(BG,\mathbb{R})\to H^*(BQ,\mathbb{R})$$
for the maps induced by $\sigma$ respectively $\pi$ in these cohomology groups.

We then have a commutative diagram
$$\xymatrix{
H^2(BG,\mathbb{R})\ar[d]^{\iota_G^*}&H^2(BQ,\mathbb{R})\ar[l]_{\pi^*_{top}}\ar[d]^{\iota_Q^*=0}\\
H^2(BG^\delta,\mathbb{R})&H^2(BQ^\delta,\mathbb{R})\ar[l]_{\pi^*_\delta} \\
H^2_b(G^\delta,\mathbb{R})\ar[u]_{\theta_G}& H^2_b(Q^\delta,\mathbb{R}).\ar[u]_{\theta_Q} \ar[l]_{\pi^*_b}^{\cong}}
$$
Let $x\in H^2(BG,\mathbb{R})$. We need to show that $\iota_G^*(x)=0$. Since $[R,R]$ is simply connected, 
$\iota_G^*(x)$ is bounded, meaning that it lies in the image of $\theta_G$ (see Theorem 1.1 of \cite{CCMP}).
By assumption, $\pi_1(Q)$ is finite. Thus $H^2(BQ,\mathbb{R})\cong\operatorname{Hom}(\pi_1(Q),\mathbb{R})=0$
which implies that $\iota_Q^*=0$ in the diagram above.
Choose $y$ such that $\theta_G  (y)=\iota_G^*(x)$. Because $y=\pi^*_b \sigma^*_b(y)$, we have
$$
\iota_G^*x=\theta_G y=\theta_G(\pi^*_b \sigma^*_b y)=\pi_\delta^*\theta_Q\sigma_b^*(y)=
\pi^*_\delta\sigma^*_\delta(\theta_G y)=
$$
$$
=\pi^*_\delta (\sigma^*_\delta \iota^*_G x)=\pi^*_\delta(\iota_Q^*\sigma^*_{top}x)=0\,,
$$
because $\iota_Q^*=0$.

{\sl{Non-split case}}. Suppose that the exact sequence $(1)$ is non-split. 
Let $\tilde{Q}\to Q$ be the universal cover.
The pull-back of $G\to Q$ over $\tilde{Q}$
yields a short exact sequence of Lie groups
\begin{equation}
R\to \bar{G}\to \tilde{Q}
\end{equation}
which is split because $\tilde{Q}$ is simply connected (see Lemma 14 of {\cite{CMPS}}). The natural
map $p:\bar{G}\to G$ is a surjective homomorphism of connected Lie groups with finite kernel $K$
isomorphic to $\pi_1(Q)$.
Since $BK$ is $\mathbb{R}$-acyclic, the induced maps 
$$p^*_{top}:H^*(BG,\mathbb{R})\stackrel{\cong}{\rightarrow} H^*(B\bar{G},\mathbb{R})\;\ \text{and}\; \
p^*_\delta: H^*(BG^\delta,\mathbb{R})\stackrel{\cong}{\rightarrow} H^*(B\bar{G}^\delta,\mathbb{R})$$ are isomorphisms.
From the {\sl{split case}} we infer that $\iota^*_{\bar{G}}: H^2(B\bar{G},\mathbb{R})\to H^2(B\bar{G}^\delta,\mathbb{R})$
is the zero map, and thus the corresponding map $\iota_G^*$ is zero too.
\end{proof}

{\bf Proof of Theorem \ref{thm: main}}.

\noindent
Let $G$ be a connected complex Lie group. Its radical $R$ is
a complex Lie subgroup and $G/R$ is complex semisimple and has therefore a finite fundamental group
\cite[Chapter XVII, Theorem 2.1]{H}. As explained in the introduction, in order to prove the theorem, it is enough to assume that $[R,R]$ is simply connected 
(in its Lie group topology) and to 
show that if $P: E\to X$ is a flat principal $G$-bundle
over a connected finite complex $X$, with classifying map $\alpha: X\to BG$,
then
there is a finite connected covering space $\beta:Y\to X$, such that the bundle $P$ pulled back to $Y$ is
trivial, i.e. such that the map $\alpha\circ\beta :Y\to BG$ is homotopic to a constant map. Let $p:G\to Q$ be the projection
and put $\gamma=Bp:BG\to BQ$. Then
the map $\gamma\circ\alpha: X\to BQ$ classifies a
principal $Q$-bundle over $X$ which is flat because $P$ is flat and the diagram
\[\xymatrix{
BG^{\delta}\ar[d]\ar[r]&BQ^{\delta}\ar[d]&\\
BG\ar[r]&BQ\\
}
\]
commutes. By Lemma \ref{semisimple case} we can find
a finite connected covering space $\delta: Z\to X$ such that the bundle classified by $\gamma\circ\alpha\circ\delta:
Z\to BQ$ is trivial. The lifting property of the fibration $BR\to BG\to BQ$
implies that   $\alpha\circ\delta: Z\to BG$ factors through $\epsilon=Bi:BR\to BG$, where
$i:R\to G$ stands for the inclusion. In other words, there is a map $\kappa:Z\to BR$,
with $\epsilon\circ\kappa: Z\to BG$ homotopic to $\alpha\circ\delta: Z\to BG$. We claim that the
(not necessarily flat) principal $R$-bundle classified by $\kappa : Z\to BR$ is virtually trivial. By Lemma \ref{R}
it suffices to show that
$$(\epsilon\circ\kappa)^*=(\alpha\circ\delta)^*=0: H^2(BG,\mathbb{R})\to H^2(Z,\mathbb{R})\;.$$
As $P$ is flat, $\alpha^*: H^2(BG,\mathbb{R})\to H^2(X,\mathbb{R})$ factors through 
$H^2(BG^\delta,\mathbb{R})$
and
since by assumption $[R,R]$ is simply connected and $Q$ is complex semisimple, Lemma \ref{G} applies and implies that
$$H^2(BG,\mathbb{R})\to H^2(BG^\delta,\mathbb{R})$$
is the zero map. Thus $(\alpha\circ\delta)^*=0$ and therefore,
by Lemma \ref{R},
the bundle classified by $\kappa: Z\to BR$ is virtually trivial.
We now choose a finite connected covering space $\mu: Y\to Z$ on which the $R$-bundle pulls back to a
trivial bundle, i.e. $\kappa\circ\mu\simeq *$. It then follows that the original $G$-bundle over $X$ pulls back
to the trivial bundle over the finite covering space $\beta=\delta\circ\mu:Y\to X$.

The following diagram, with commuting squares up to homotopy, depicts, for the convenience of the reader, the maps described above:

$$\xymatrix{
Y\ar[d]^\mu\ar[r]\ar@/_1pc/[dd]_\beta&\{*\}\ar[d]&\\
Z\ar[r]^\kappa\ar[d]^\delta&BR\ar[d]^\epsilon&\\
X\ar[r]^\alpha&BG\ar[r]^\gamma&BQ\;.
}
$$\noindent
\hfill$\square$

\section{Proof of Corollary \ref{maincoro}}\label{sec: proof of the corollary}

We first recall some facts on the complexification of a connected
Lie group $G$. We follow the notation used in Hochschild {\cite{Hochschild}, 
(see also Bourbaki \cite{Bou}*{Chapter III, \S 6, Prop. 20}).
To any Lie group corresponds a  complex Lie group $G^+$ and
a homomorphism of Lie groups $$\gamma_G: G\to G^+,$$ called the universal complexification
of $G$,
with the property that, for every continuous homomorphism $\eta$ of $G$ into a complex Lie group $H$, there is one and only one complex analytic homomorphism
$\eta^+:G^+\to H$ such that $\eta=\eta^+\gamma_G$. 
In general $\gamma_G$ is
not injective. Its kernel is a central (not necessarily discrete) subgroup of $G$.
Let $R<G$ denote the radical of
the connected Lie group $G$ and
$L<G$ a Levi subgroup (a maximal connected semisimple subgroup). Then $G={RL}$ and in case $L<G$ is closed,
the kernel of $\gamma_G$ coincides with the kernel of $\gamma_L$ and is discrete in $G$ (see \cite{Hochschild}*{Theorem 4}). Also,
if $G$ is linear, $\gamma_G$ is injective and for $G$ compact, $\gamma_G$ maps $G$ isomorphically
onto a maximal compact subgroup of $G^+$. Therefore, for compact $G$, the map $\gamma_G$ is a homotopy
equivalence. As explained in \cite{Hochschild}, in the case $R$ is a connected solvable Lie group (not necessarily linear), $\gamma_R$ is
injective and induces an isomorphism between fundamental groups $\pi_1(R)\to\pi_1(R^+)$.
The universal covers of the solvable Lie groups $R$ and $R^+$ being contractible, it follows that $\gamma_R: R\to R^+$
is a homotopy equivalence.

\begin{lemm}{\label{amenablecase}}
Let $G$ be a connected amenable Lie group. Then the complexification map
$\gamma_G: G\to G^+$ is one-to-one and a homotopy equivalence.
\end{lemm}
\begin{proof}
A connected Lie group $G$ is amenable if and only if it fits
	in a short exact sequence 
	$$\{1\}\to R\to G\to Q\to\{1\},$$
where $R$ denotes the radical of $G$ and the quotient $Q$ is compact semisimple \cite{Zimmer}*{Corollary 4.1.9}.	
Let $L<G$ be a Levi subgroup. Since $G/R=Q$ is compact
and semisimple, its fundamental group is finite. 
Thus $L\to Q$, induced by the projection $G\to Q$, is a finite covering
space
and it follows that $L$ is compact, thus closed in $G$. Moreover, $L$ is linear
and we conclude that $\gamma_L: L\to L^+$ is one-to-one. According to \cite[Theo\-rem 4, $(2)\Leftrightarrow (5)$]{Hochschild},
we conclude that
$\gamma_G$ is injective too. 
Consider the commutative diagram
$$\xymatrix{
R \ar@{^{(}->} [d]^{\gamma_R}_\simeq\ar[r]^{\iota}&G\ar[r]^{\pi}\ar@{^{(}->}[d]^{\gamma_G}&Q\ar@{^{(}->}[d]^{\gamma_Q}_\simeq\\
\;R^+\ar[r]^{\iota^+}&\;G^+\ar[r]^{\pi^+}&\;Q^+\;.\\
}
$$
As remarked above, $\gamma_R$ and $\gamma_Q$ are injective maps and homotopy equivalences.
By \cite{Hochschild}*{Theorem 4}, $\iota^+$ maps $R^+$ isomorphically onto the
radical of $G^+$. We claim that $G^+/R^+$
it is isomorphic to $Q^+$. To see this, we need to verify that this quotient has the universal 
property of $Q^+$. Let $\nu: G\to G^+\to G^+/R^+$ be the natural map. Since $R\subset\ker\nu$, we obtain
an natural map $\overline{\nu}: Q\to G^+/R^+$. Let $f:Q\to C$ be an
analytic homomorphism into a complex Lie group C. Then $f\circ\pi: G\to Q\to C$ is trivial on 
$R$ and extends therefore uniquely to a complex analytic homomorphism $G^+\to C$
which vanishes on $R^+$. It follows that the original map $f$ factors uniquely through
$\overline{\nu}: Q\to G^+/R^+$, showing that $G^+/R^+\cong Q^+$.
Note that both horizontal lines in the diagram above are fibration sequences.
We conclude that $\gamma_G$ must be a homotopy equivalence too.
\end{proof}

The following lemma is a general fact about universal complexifications of connected solvable Lie groups. It will be useful
in the proof of Corollary \ref{maincoro}.

\begin{lemma}\label{lem: + and [] commute} Let $R$ be a connected solvable Lie group. Then 
	\[
		[R,R]^+=[R^+,R^+].
	\]
That is, the universal complexification of the derived subgroup of $R$ is isomorphic
to the derived subgroup of the universal complexification of $R$.	
\end{lemma}

\begin{proof}
	As $R$ is solvable, $\gamma_R$ is one-to-one hence so is its restriction
	$\eta: [R,R]\to [R^+,R^+]$  to $[R,R]$. The universal property of 
	\[
		\gamma_{[R,R]}:[R,R]\to [R,R]^+,
	\]  
	implies the existence of a complex analytic homomorphism 
	\[
	   \eta^+:[R,R]^+\to [R^+,R^+], 
	\]
	such that $\eta^+\gamma_{[R,R]}=\eta$. Taking derivatives at the identities
	and using the fact that for any real Lie algebra $\frak{r}$ we have 
	\begin{equation}\label{equ: equation}
		[\frak{r}\otimes\mathbb C,\frak{r}\otimes\mathbb C]=[\frak{r},\frak{r}]\otimes\mathbb C,
	\end{equation} 
	we deduce that $\eta^+$ is a local isomorphism, hence a covering homomorphism.
	This proves the lemma in the case $R$ is simply connected. Indeed, 
	the inclusion $\gamma_R: R\to R^+$ is a homotopy equivalence, hence $R^+$ is also simply connected, and according to 
	\cite[Lemma 6]{CMPS} we have
	$$\pi_1([R^+,R^+])=\pi_1(R^+)\cap[R^+,R^+].$$
	This shows that $[R^+,R^+]$ is also simply connected, hence                                                                          
	$\eta^+$ is a global isomorphism.  
	To handle the general case, let us show that the discrete kernel of $\eta^+$ is trivial.
	To that end, we show that  the natural embeddings of fundamental
	groups in the centers of universal covers coincide. Let $\tilde{R}$ be the universal cover of $R$. It is obvious from the construction of the universal complexification that the universal cover $\tilde{R^+}$ of $R^+$ coincides with $(\tilde{R})^+$. We have:
	\begin{align*}
		   \pi_1([R,R]^+)=&\pi_1([R,R])=\pi_1(R)\cap[\tilde{R},\tilde{R}]\\
		                 =&\pi_1(R)\cap[\tilde{R},\tilde{R}]^+=\pi_1(R^+)\cap[\tilde{R},\tilde{R}]^+\\
		=&\pi_1(R^+)\cap[(\tilde{R})^+,(\tilde{R})^+]=\pi_1(R^+)\cap[\tilde{R^+},\tilde{R^+}]\\
=&\pi_1([R^+,R^+]).		
	\end{align*}
	The first equality (as well as the fourth one) is true because the embedding of a connected solvable Lie group in its universal
	complexification is a homotopy equivalence, the second equality (as well as the last one) is a general fact
	(see \cite[Lemma 6]{CMPS}) about
	closed normal subgroups in Lie groups, the third equality follows from the fact that 
	$[\tilde{R},\tilde{R}]^+\cap\tilde{R}\subset [\tilde{R},\tilde{R}]$ which is deduced
	from the corresponding inclusion between Lie algebras. The fifth equality is true because we have already proved the lemma for simply connected solvable Lie groups hence $[\tilde{R},\tilde{R}]^+=[(\tilde{R})^+,(\tilde{R})^+]$. 
\end{proof}

{\bf Proof of Corollary \ref{maincoro}.}
As explained in the introduction, in order to prove that the four conditions are equivalent,
it is enough to show that if $G$ is a  connected amenable Lie group with radical $R$ such that
$[R,R]$ is simply connected, then any flat principal $G$-bundle over a finite $CW$-complex is virtually trivial.
As we have observed earlier,
the complexfication map $R\to R^+$ is injective and, applying again \cite[Theo\-rem 4, $(2)\Leftrightarrow (5)$]{Hochschild}, we see that as $G/R$ is compact semisimple, $R^+$ maps isomorphically onto the radical of $G^+$. 
As we have seen in the course of the proof of Lemma \ref{amenablecase},
the map $\gamma_G$ restricted to $R$ agrees with $\gamma_R$. Thus, as $[R,R]$ is simply connected by hypothesis, and as $[R^+,R^+]/\gamma_G[R,R]$ is simply connected according 
to \cite[Theorem 3]{Hochschild} and Lemma \ref{lem: + and [] commute}, we deduce that 
$[R^+,R^+]$ is simply
connected too. Let $f:X\to BG$ classify a flat principal $G$-bundle over the finite connected
$CW$-complex $X$. Since $[R^+,R^+]$ is the commutator subgroup of the radical of $G^+$
and $[R^+,R^+]$ is simply connected, we conclude by Theorem \ref{thm: main} that the flat bundle classified by $B\gamma_G\circ f:
X\to BG^+$ is virtually trivial. Because $B\gamma_G$ is a homotopy equivalence (Lemma \ref{amenablecase}), the bundle
classified by $f$ is virtually trivial too.
\hfill$\square$

\end{document}